\documentclass[12pt]{amsart}
\usepackage{amsmath,amsthm,amsfonts,amssymb,eucal}
\usepackage[notcite,notref]{showkeys}

\renewcommand {\a}{ \alpha }

\newcommand{\y}{\eta}
\newcommand{\e}{\epsilon}

\newcommand{\g}{\gamma}
\newcommand{\G}{\Gamma}

\newcommand{\D}{\Delta}
\newcommand{\s}{\sigma}

\newcommand{\N}{ \mathbb N}

\newcommand{\CC}{\mathcal C}
\newcommand{\CD}{\mathcal D}

\newcommand{\CH}{\mathcal H}

\newcommand{\CX}{\mathcal X}

\newcommand{\CN}{\mathcal N}

\newcommand {\GH}{\mathfrak H}

\newcommand {\ba}{\mathbf a}

\newcommand {\bh}{\mathbf h}

\newcommand {\BA}{\mathbf A}

\newcommand {\BH}{\mathbf H}

\newcommand {\BT}{\mathbf T}

\newcommand{\wt}{\widetilde}

\DeclareMathOperator{\res}{\restriction}

\DeclareMathOperator{\tr}{Tr}

 \DeclareMathOperator{\gen}{Gen}

\newtheorem{thm}{Theorem}[section]

\newtheorem{prop}[thm]{Proposition}

\theoremstyle{definition}
\newtheorem{example}[thm]{Example}
\theoremstyle{remark}

\numberwithin{equation}{section}

\newcommand{\thmref}[1]{Theorem~\ref{#1}}

\newcommand{\bsymb}{\boldsymbol}

\newcommand{\vs}{\vskip0.2cm}

\begin{document}

\title[negative eigenvalues of Schr\"odinger operator on graphs]
{Remarks on counting negative eigenvalues of Schr\"odinger
operator on regular metric trees}

\dedicatory{To Victor Petrovich Khavin, a friend and colleague, on
his 75-th birthday}

\author{ Michael Solomyak}

\address{Department of Mathematics
The Weizmann Institute of Science
Rehovot 76100 Israel}
\email{michail.solomyak@weizmann.ac.il}

\begin{abstract}
We discuss estimates on the number $N_-(\BH_{\a V})$ of negative eigenvalues of the Schr\"odinger operator
$\BH_{\a V}=-\D-\a V$
on regular metric trees, as depending on the properties of the potential $V\ge 0$ and on the value of the large parameter $\a$. We obtain conditions on $V$ guaranteeing the behavior $N_-(\BH_{\a V})=O(\a^p)$ for any given $p\ge 1/2$.
For a special class of trees we show that these conditions are not only sufficient but also necessary.
For $p>1/2$ the order-sharp estimates involve a (quasi-)norm of $V$ in some `weak' $L_p$- or $\ell_p(L_1)$-space. We show that
the results obtained can be easily derived from the ones of \cite{NS}.

The results considerably improve the estimates found in the recent paper \cite{EFK}.

\end{abstract}

\date{\today}

\maketitle

\section{Introduction}\label{intro}
\subsection{Preliminaries.}
In this note we discuss estimates for the number $N_-(\BH_V)$ of
the negative eigenvalues of the Schr\"odinger operator
$\BH_V=-\D-V$ on regular metric trees. The classical Birman -- Schwinger principle allows one
to reduce such estimates to the study of the spectrum of a certain compact, self-adjoint
operator acting in an appropriate Hilbert space; see Section \ref{prfs} below for an explanation
of this reduction. The detailed study of this spectrum was initiated by
Naimark and the author \cite{NS}, and some estimates for  $N_-(\BH_V)$ immediately follow from there.

Further results on the behavior of $N_-(\BH_V)$ were obtained
recently by Ekholm, Frank, and Kova\v{r}\'{i}k \cite{EFK}. The
main novelty in \cite{EFK} is an analysis of the connection
between the estimates for $N_-(\BH_V)$ and the {\it global
dimension} of a regular tree. This is an important step in the
understanding of spectral properties of the Schr\"odinger operator
on trees.

The proofs in \cite{EFK} are based upon a scheme
developed by Naimark and Solomyak \cite{NS, NS1}, and independently by Carlson \cite{C}.
Our main goal in this note is to
show that the estimates obtained in \cite{EFK} can be considerably improved by the
direct reduction to the results of \cite{NS}. We also show that
the estimates in terms of the `weak $L_p$-spaces'
obtained by this approach, are order-sharp in the strong coupling
limit. This means that when applied to the operator
$\BH_{\a V}$, where $\a>0$ is a large parameter, these estimates are of the
type
\begin{equation}\label{estim}
    N_-(\BH_{\a V})=O(\a^p),\qquad \a\to\infty,
\end{equation}
where `$O$' cannot be replaced by `$o$'. On the contrary, in the
estimates obtained in \cite{EFK} in terms of the standard $L_p$,
such replacement is always possible.

We also discuss (in \thmref{th:p}) estimates of a different nature,
where the function $N_-(\BH_{\a V})$
is evaluated in terms of a certain number sequence associated with the potential.
In \thmref{th:pb} we single out a sub-class of regular trees,
for which the conditions guaranteeing the order $N_-(\BH_{\a
V})=O(\a^p)$ with $p>1/2$ are not only sufficient, but also necessary.

Any metric tree has local dimension one. As a consequence, for the
rapidly decaying potentials $V$ one always has
 $N_-(\BH_{\a V})=O(\a^{1/2})$. For the slowly decaying
 potentials, such that \eqref{estim} is satisfied with some
 $p>1/2$, the estimates in $L^p$-terms are never order-sharp: any
 order-sharp estimate must involve a norm, or a quasi-norm, of some
 non-separable function space. In this connection, see a discussion in
 \cite{RS}. The results of the present paper once
more substantiate this claim.\vs

\subsection{Geometry of a tree.}\label{prel}
Let $\G$ be a rooted metric tree, with the root $o$. We suppose
that the number of vertices and the number of edges are infinite,
and that each edge has finite length.
The distance $\rho(x,y)$ and the partial ordering $x\preceq y$ on
$\G$ are introduced in a natural way, and we write $x\prec y$ if
and only if $x\preceq y$ and $x\neq y$. We denote $|x|=\rho(o,x)$.
On each edge $e\subset\G$ the partial ordering turns into the
standard ordering defined by the inequality $|x|\le|y|$. With each
vertex $v\in\G$ we associate its {\it generation} $\gen(v)$, i.e.
the number of vertices $w$ such that $o\preceq w\prec v$. In
particular, $\gen(o)=0$. The {\it branching number} $b(v)$ of a
vertex $v$ is the number of edges emanating from $v$. We suppose
that $b(o)=1$ and that $1<b(v)<\infty$ for any vertex $v\neq o$.

In this paper we consider the regular trees, see
\cite{NS,NS1} for more detail. A tree $\G$ is said to be {\it
regular} (or {\it symmetric}, or {\it radial}), if all
vertices of the same generation have equal branching numbers and
lie on the equal distance from the root. Each regular tree is
completely determined by specifying two number sequences,
$\{b_n\}$ and $\{t_n\}$, where $b_n=b(v)$ and $t_n=|v|$ for any
vertex $v$ of generation $n,\ n=0,1,\ldots$. We suppose that the
sequence $\{t_n\}$ is unbounded. The {\it
branching function} of $\G$, defined as
\begin{equation}\label{g}
    g_0(t)= b_0b_1\ldots b_n\ {\text{for}}\ t_n< t\le t_{n+1},\qquad n=1,2,\ldots
\end{equation}
is an important characteristic of a regular tree. \vs

The natural measure $dx$ on $\G$ is induced by the Lebesgue
measure on the edges. The spaces $L_p(\G)$ are understood as the
Lebesgue spaces with respect to this measure. The norm in
$L_p(\G)$ is denoted by $\|\cdot\|_p$; for $p=2$, we write simply
$\|\cdot\|$.

\subsection{Sobolev spaces on a tree}\label{sob}
Differentiation, in the direction compatible with the ordering on
$\G$, is well-defined at any point $x\in\G$ except for the
vertices. The Sobolev space $H^1(\G)$ consists of all continuous
functions $u$ on $\G$, such that $u\res e\in H^1(e)$ on each edge
$e\subset \G$, and the condition
\[ \|u\|^2_{H^1(\G)}:=\int_\G(|u'(x)|^2+|u(x)|^2)dx<\infty\]
is satisfied. It is easy to show that any function $u\in H^1(\G)$
vanishes as $|x|\to\infty$, see, e.g., Lemma 3.1 in \cite{S}. It
immediately follows that the functions $u\in H^1(\G)$ with bounded
support are dense in $H^1(\G)$. The boundary condition $u(o)=0$
selects the subspace
\[H^{1,0}(\G)\subset H^1(\G).\]
The functions $u\in H^{1,0}(\G)$ with bounded support are dense in
this subspace.

Along with $H^{1,0}(\G)$, we need also the {\it homogeneous}
Sobolev space $\CH^{1,0}(\G)$ formed by all continuous
functions $u$ on $\G$, such that $u(o)=0$, $u\res e\in H^1(e)$ on
each edge $e\subset \G$, and $u'\in L_2(\G)$. By definition,
\[ \|u\|_{\CH^1(\G)}=\|u'\|.\]

In contrast with the space $H^{1,0}(\G)$, the functions $u\in
\CH^{1,0}(\G)$ with bounded support are not always dense in
$\CH^{1,0}(\G)$. As a consequence, the set $H^{1,0}(\G)$, which is
evidently contained in $\CH^{1,0}(\G)$, is not always dense in the
latter space. This property depends on geometry of the tree. A
tree such that $H^{1,0}(\G)$ is dense in $\CH^{1,0}(\G)$ is called
{\it recurrent}, otherwise it is called {\it transient}. We shall
denote the closure of $H^{1,0}(\G)$ in $\CH^{1,0}(\G)$ by
$\CH^\circ(\G)$. So, $\CH^\circ(\G)=\CH^{1,0}(\G)$ if and only if
the tree $\G$ is recurrent.

A simple necessary and sufficient condition of transiency is known
for the regular trees, see, e.g., \cite{NS1}, Theorem 3.2. Namely,
such tree is transient if and only if its {\it reduced height}
\begin{equation}\label{trans}
    l(\G):=\int_\G \frac{dt}{g_0(t)}
\end{equation}
is finite.

\section{Schr\"odinger operator on $\G$.}\label{schr}
Our main object in this paper is the Schr\"odinger operator on $\G$,
\begin{equation}\label{sch}
\BH_{\a V}=-\D_D-\a V,
\end{equation}
with the non-negative potential $V(x)$ and a large
parameter ({\it the coupling constant}) $\a>0$. In \eqref{sch} the
symbol $\D_D$ stands for the Dirichlet Laplacian on $\G$. In
Section \ref{neum} we discuss also an analogue of \eqref{sch} for the Neumann
Laplacian $\D_N$.

The operator \eqref{sch} is defined via its quadratic form which
is
\begin{equation}\label{shrqf}
\bh_{\a V}[u]=\int_\G (|u'|^2-\a V|u|^2)dx,\qquad u\in
H^{1,0}(\G).
\end{equation}

We are interested in estimating the number of the negative
eigenvalues of the operator $\BH_{\a V}$ in terms of the potential
$V$ and the coupling constant $\a$.\vs

Assume that the potential $V\ge 0$ is such that the inequality
\begin{equation}\label{hw}
    \int_\G V|u|^2dx\le C_V\int_\G|u'|^2dx,\qquad\forall u\in
    H^{1,0}(\G)
\end{equation}
is satisfied. By the continuity, it extends to all
$u\in\CH^\circ(\G)$. We assume that $C_V$ is the least possible
constant in \eqref{hw}. We call any such potential $V$ {\it a
Hardy weight on $\CH^\circ(\G)$}, and we say that the Hardy weight
is {\it normalized}, if $C_V=1$.  If $V$ is a Hardy weight, the
operator \eqref{sch} is well-defined at least for small values of
$\a$. Some further assumptions about $V$, that we impose later,
imply finiteness of the negative spectrum of this operator.

In general, for a self-adjoint, bounded from below operator $\BH$
in a Hilbert space $\GH$, whose negative spectrum is finite, we
denote by $N_-(\BH)$ the number of its negative eigenvalues
counted according to their multiplicities. \vs

For estimating the quantity $N_-(\BH_{\a V})$ it is important to
have a description of the Hardy weights on $\CH^\circ(\G)$. Such
description, for general trees and the Hardy weights on a wider
class $\CH^{1,0}(\G)$, was obtained in \cite{EHP}. A similar
result for $\CH^\circ(\G)$ seems to be unknown so far. The
situation changes if we restrict ourselves to the regular trees
and to the {\it symmetric} weights
\begin{equation*}
    V(x)=v(|x|),\qquad \forall x\in\G
\end{equation*}
where $v\ge 0$ is a measurable function on $[0,\infty)$. For the
regular trees an exhaustive description of all symmetric Hardy
weights can be easily derived from the classical Hardy-type
inequalities, see, e.g., \cite{OK}, or Section 1.3 in the book
\cite{M}. The result of such reduction was presented in
\cite{NS1}, Theorem 5.2, but only for the transient regular trees.
We reproduce it below, including also the recurrent case. The
proof for this case remains the same.

\begin{prop}\label{prhardy} Let $\G$ be a regular tree, and let
$g_0(t)$ be its branching function \eqref{g}. Suppose
$V(x)=v(|x|)\ge 0$ is a symmetric weight function on $\G$.

1$^\circ$ Let $\G$ be recurrent. Then $V(x)$ is a Hardy weight on
$\CH^\circ(\G)$ if and only if
\begin{equation*}
    B_0(v):=\sup_{t>0}\left(\int_t^\infty v(s)g_0(s)ds\cdot\int_0^t
    \frac{ds}{g_0(s)}\right)<\infty.
\end{equation*}
Moreover, the least possible constant $C_V$ in \eqref{hw}
satisfies
\[ B_0(v)\le C_V\le 4B_0(v).\]

2$^\circ$ Let $\G$ be transient. Then $V(x)$ is a Hardy weight on
$\CH^\circ(\G)$ if and only if the following two conditions are
satisfied:
\begin{equation}\label{hardy}
B_1(v):=\sup\limits_{t>t_1}\left(\int_{t_1}^t v(s)g_0(s)ds\cdot
\int_t^\infty\frac{ds}{g_0(s)}\right)<\infty;
\end{equation}
\begin{equation}\label{hardy1}
B_2(v):=\sup_{t<t_1}\left(t\cdot\int_t^{t_1}v(s)ds\right)<\infty.
\end{equation}
Moreover,
\begin{gather*}
C_V\le 4(B_1(v)+B_2(v));\qquad
B_1(v)\le C_V;\
B_2(v)\le\left(1+\frac{b_1t_1}{t_2-t_1}\right)C_V.
\end{gather*}
\end{prop}
\section{Classes of functions, of number sequences, and of compact
operators}\label{clas} As we shall see, the sharp estimates of the
function $N_-(\BH_{\a V})$ involve the {\it weak $L_p$-spaces} on
the tree, rather than the classical Lebesgue spaces $L_p$. Here we recall
their definitions. See, e.g., \cite{BL}, Section 1.3, for more
detail.

Let $(\CX,\mu)$ be a measure space with $\s$-finite measure. A
measurable function $f$ on $\CX$ belongs to the class
$L_{p,w}(\CX,\mu),\ 0<p<\infty$, if and only if
\begin{equation}\label{weak}
\|f\|_{L_{p,w}(\CX,\mu)}:=\sup_{t>0}\left(
t\cdot\left(\mu\{x\in\CX:|f(x)|>t\}\right)^{1/p}\right)<\infty.
\end{equation}
 The spaces $L_{p,w}$ are linear. The functional
\eqref{weak} defines a {\it quasi-norm} on $L_{p,w}$, and the
space is complete with respect to this quasi-norm. If $p>1$, and
only in this case, a norm equivalent to the quasi-norm
\eqref{weak} does exist. However, in various estimates it is more
convenient to use the quasi-norm \eqref{weak}.

If the measure $\mu$ does not reduce to the sum of a finite number
of atoms, the space $L_{p,w}(\CX,\mu)$ is non-separable. The
condition
\[ \mu\{x\in\CX:|f(x)|>t\}=o(t^{-p})\qquad {\text{as}}\ t+t^{-1}\to \infty\]
singles out a separable subspace in $L_{p,w}(\CX,\mu)$, which we
denote by $L_{p,w}^\circ(\CX,\mu)$. It is well-known (and easy to
check) that
\begin{equation*}
  \|f\|^p_{L_{p,w}(\CX,\mu)}\le \int_\CX|f|^p d\mu;\qquad
   L_{p}(\CX,\mu)\subset L_{p,w}^{\circ}(\CX,\mu).
\end{equation*}\vs

In particular, let $\CX=\G$ be a metric tree, and let $\mu$ be
the measure generated by a weight function $\Phi(x)\ge 0$ on $\G$:
\[ d\mu=\Phi(x)dx.\]
We denote the corresponding space by $L_{p,w}(\G,\Phi)$ and the
functional \eqref{weak} by $\|f\|_{p,w;\Phi}$. So,
\begin{equation*}
    \|f\|_{p,w;\Phi}^p=\sup_{t>0}\left(t^p\int\limits_{|f(x)|>t}
    \Phi(x)dx\right).
\end{equation*}

We also need the weak $\ell_p$-spaces of number sequences. These
are a particular case of the spaces $L_{p,w}(\CX,\mu)$, where
$\CX=\N$, or $\CX=\CN_0:=\{0,1,\ldots\}$. The measure $\mu$ is
either the standard counting measure, or it is given by a sequence
of positive weights $\Phi_n=\mu(\{n\})$. For a sequence
$\bsymb{f}=\{f_n\}_{n\in\CX}$ we have
\begin{equation*}
    \|\bsymb{f}\|_{p,w;\Phi_n}^p=\sup_{t>0}\left(t^p
    \sum\limits_{|f(n)|>t}\Phi_n\right).
\end{equation*}

The case of counting measure is especially simple; here we drop
the symbol $\Phi_n$ in the notation. Consider the sequence
$\{f^*_n\}$ whose terms are the numbers $|f_n|$, rearranged
in the non-increasing order. Then $\bsymb{f}\in \ell_{p,w}$ if and
only if $f^*_n=O(n^{-1/p})$, and $\bsymb{f}\in \ell_{p,w}^\circ $
if and only if $f^*_n=o(n^{-1/p})$. Moreover,
\[ \|\bsymb{f}\|_{p,w}=\sup_n(n^{1/p}f_n^*).\]
\vs

Now, let $\BA$ be a compact operator in a Hilbert space $\GH$, and
let $\{s_n(\BA)\}$ be the sequence of its singular numbers, counted according to
their multiplicities. Recall that for a non-negative self-adjoint
operator its singular numbers coincide with the eigenvalues. By
definition, $\BA$ belongs to the weak Schatten class $\CC_{p,w}$,
if $\{s_n(\BA)\}\in\ell_{p,w}$, and
\[ \|\BA\|_{p,w}:=\|\{s_n(\BA)\}\|_{p,w}=\sup_n
\left(n^{1/p}s_n(\BA)\right).\] The class $\CC^\circ_{p,w}$ is
defined as the set of all compact operators such that
$s_n(\BA)=o(n^{-1/p})$.

For a compact operator $\BA$ we denote
\[ n(s;\BA)=\#\{n:s_n(\BA)>s\},\qquad s>0.\]
 Evidently,
the inclusion $\BA\in \CC^p_w$ is equivalent to the inequality
\[n(s; \BA)\le C n^{-p},\qquad\forall n\in\N,\]
and moreover, the best possible constant here is
$C=\|\BA\|_{p,w}^p.$ Also, $\BA\in \CC^\circ_{p,w}$ if and only if
$n(s; \BA)=o(n^{-p})$.
\section{Main results on the behavior of $N_-(\BH_{\a
V})$}\label{main} Here we present the basic estimates for the
number of negative eigenvalues of the operator $\BH_{\a V}$. Their
proofs are given in Section \ref{prfs}.

\begin{thm}\label{genthm}
Let $\G$ be a regular metric tree and let a normalized and
symmetric Hardy weight $\Psi(x)=\psi(|x|)>0$ on $\CH^\circ(\G)$ be
chosen. Define the weight function
    $\Phi(x)=|x|\Psi(x)$.
Suppose $V$ is a non-negative potential on $\G$ such that
$V\Psi^{-1}\in L_{p,w}(\G,\Phi)$ for some $p>1$. Then
\begin{equation}\label{bas}
    N_-(\BH_{\a V})
    \le C\a^p\sup_{t>0}
    \left(t^p\int\limits_{V(x)>t\Psi(x)}\Phi(x)dx\right),
    \qquad\forall\a>0,
\end{equation}
and
\begin{equation}\label{lsup}
    \limsup_{\a\to \infty}\a^{-p}N_-(\BH_{\a V})
    \le C\limsup_{t+t^{-1}\to\infty}
    \left(t^p\int\limits_{V(x)>t\Psi(x)}\Phi(x)dx\right).
\end{equation}
In particular, if
    $\int\limits_{V(x)>t\Psi(x)}\Phi(x)dx=o(t^{-p})$ as
    $t+t^{-1}\to\infty$, then
\begin{equation}\label{osmall}
 N_-(\BH_{\a V})=o(\a^p).
\end{equation}

If $V\Psi^{-1}$ belongs to the narrower space $L_{p}(\G,\Phi)$,
then
\begin{equation}\label{narr}
    N_-(\BH_{\a V})\le C\a^p\int_\G|x|V^p\Psi^{1-p}dx,
    \end{equation}
and \eqref{osmall} is satisfied.
\end{thm}

\subsection{Global dimension of a regular tree}\label{regtr}
Suppose that the branching function $g_0(t)$
satisfies the two-sided inequality
\begin{equation*}
    c_1(1+t)^{d-1}\le g_0(t)\le c_2(1+t)^{d-1},\ 0<c_1<c_2,\qquad
    \forall t>0,
\end{equation*}
with some $d\ge 1$. Following \cite{EFK}, we shall say that the
number $d$ is the {\it global dimension} of the regular tree $\G$.
This is an important invariant of a tree. The existence of a
global dimension is itself a serious restriction on the structure
of the tree.

It is more convenient for us to use an equivalent definition of
global dimension:
\begin{equation}\label{dimtree1}
    c'_1t^{d-1}\le g_0(t)\le c'_2t^{d-1},\ 0<c'_1\le c'_2,\qquad
    \forall t>t_1.
\end{equation}
A direct inspection shows that the function $|x|^{-2}$ is a
symmetric Hardy weight on $\CH^\circ(\G)$, provided that $\G$ has
global dimension $d\neq 2$. Taking $\Psi(x)=c|x|^{-2}$ with an
appropriate $c$ (and hence, $\Phi(x)=c|x|^{-1}$) and applying \thmref{genthm}, we come to the
following result.

\begin{thm}\label{special}
Let $\G$ be a regular metric tree of global dimension $d\neq 2$.
Suppose the potential $V\ge 0$ is such that $V|x|^2\in
L_{p,w}(\G,|x|^{-1})$ for some $p>1$. Then
\begin{equation}\label{bas1}
      N_-(\BH_{\a V})\le C\a^p\sup_{t>0}
    \left(t^p\int\limits_{V(x)|x|^2>t}|x|^{-1}dx\right)
\end{equation}
and
\begin{equation}\label{lsup1}
    \limsup_{\a\to \infty}\a^{-p}N_-(\BH_{\a V})
    \le C\limsup_{t+t^{-1}\to\infty}
    \left(t^p\int\limits_{V(x)|x|^2>t}|x|^{-1}dx\right).
\end{equation}
In particular, if
    $\int\limits_{V(x)|x|^2>t}|x|^{-1}dx=o(t^{-p})$ as
    $t+t^{-1}\to\infty$, then \eqref{osmall} is satisfied.
If $V|x|^2\in L_{p}(\G,|x|^{-1})$, then along with \eqref{osmall} we have
\begin{equation}\label{narr1}
    N_-(\BH_{\a V})\le C\a^p\int_\G V^p|x|^{2p-1}dx.
\end{equation}
\end{thm}\vs
Any metric tree has local dimension one, and hence, for the fast
decaying potentials one must have $N_-(\BH_{\a V})=O(\a^{1/2})$.
At the same time, both Theorems \ref{genthm} and \ref{special}
describe only the situation where $N_-(\BH_{\a V})=O(\a^p)$ with
$p>1$, which corresponds to the potentials that decay rather
slowly. The next result concerns the case $1/2<p\le 1$. On the
other hand, it applies only to the symmetric potentials
$V(x)=v(|x|)$. We also suppose $V\in L_1(\G)$. Actually, the
latter is a restriction on the behavior of $V$ only near the root.
It is necessary in case of the Neumann boundary condition, see
Section \ref{neum}, but not in case of the Dirichlet condition.

The properties of the potential in the next two theorems are
expressed in terms different from those in Theorems \ref{genthm}
and \ref{special}. With any symmetric potential $V(x)=v(|x|)\ge 0$
from $L_1(\G)$ we associate the sequence
$\bsymb\y(V)=\{\y_n(V)\}$, where
\begin{equation*}
\y_n(V)=t_{n+1}\int_{t_n}^{t_{n+1}}v(t)dt, \qquad n=0,1,\ldots
\end{equation*}

\begin{thm}\label{th:p} Let $V\in L_1(\G)$ be a symmetric, non-negative
potential on a regular tree $\G$.

1$^\circ$ Suppose the sequence  $\bsymb\y(V)=\{\y_n(V)\}$ belongs to the space $\ell_{1/2}$
with respect to the weight sequence $\{g_0(t_{n+1})\}$. Then the estimate
\begin{equation}\label{1/2}
N_-(\BH_{\a V})\le C\a^{1/2}\sum_n \y^{1/2}_n(V)g_0(t_{n+1})
\end{equation}
and the Weyl type asymptotic
formula
\begin{equation}\label{w}
    \lim_{\a\to\infty}\a^{-1/2}N_-(\BH_{\a V})=\frac1{\pi}\int_\G
    V^{1/2}(x)dx=\frac1{\pi}\int_0^\infty v^{1/2}(t)g_0(t)dt.
\end{equation}
are satisfied.

2$^\circ$  Suppose
 $\bsymb\y(V)\in\ell_{p,w}(\N_0;g_0(t_{n+1}))$ for some
 $p\in(1/2,1 )$. Then
\begin{equation}\label{<1}
     N_-(\BH_{\a V})\le
C\a^p\sup_{t>0}\left(t^p\sum_{\y_n(V)>t}g_0(t_{n+1})\right)
\end{equation}
and
\begin{equation}\label{lsup3}
    \limsup_{\a\to \infty}\a^{-p}N_-(\BH_{\a V})
    \le C\limsup_{t\to 0}
    \left(t^p\sum_{\y_n(V)>t}g_0(t_{n+1})\right).
\end{equation}
In particular, if
    $\sum_{\y_n(V)>t}g_0(t_{n+1})=o(t^{-p})$ as $t\to 0$, then \eqref{osmall}
is satisfied.

If $\bsymb\y(V)\in\ell_p(\N_0;g_0(t_{n+1}))$ for some $p\in
(1/2,1]$, then along with \eqref{osmall} we have
\begin{equation}\label{str<1}
    N_-(\BH_{\a V})\le C\a^p\sum_{n=0}^\infty\y^p_n(V)g_0(t_{n+1}).
\end{equation}
\end{thm}

\subsection{Estimates for the $b$-regular trees}\label{breg} More
advanced results can be obtained if we suppose that for all the
vertices $v\neq o$ the branching number is the same
integer $b>1$. In other words, we suppose that
\[ b_n=b,\qquad \forall n>0.\]
In the paper \cite{NS}, Section 7.2, such trees were called {\it
$b$-regular}. For a $b$-regular tree one has
\begin{equation}\label{b-reg}
g_0(t)=b^n,\qquad t_n< t\le t_{n+1}.
\end{equation}
If, besides, $\G$ has global dimension $d$, then \eqref{b-reg},
together with \eqref{dimtree1}, implies (as $t\to t_n+$) that
\begin{equation}\label{reg1}
    c'_1t_n^{d-1}\le b^n< c'_2t_n^{d-1}.
\end{equation}
Hence, the sequence $\{t_n\}$ grows as a geometric progression.

For the $b$-regular trees the result of
\thmref{th:p} can be considerably strengthened: it extends to arbitrary
values of $p>1/2$ and, most importantly, the condition
\[\bsymb\y(V)\in\ell_{p,w}(\N_0,g_0(t_{n+1}))=\ell_{p,w}(\N_0,b^n)\]
turns out to be not only sufficient, but also necessary for
 $N_-(\BH_{\a V})=O(\a^p)$.
\begin{thm}\label{th:pb}
Let $V\in L_1(\G)$ be a symmetric, non-negative potential on a
$b$-regular tree $\G$. Suppose also that $\G$ has global dimension
$d\neq 2$.

1$^\circ$ Suppose
 $\bsymb\y(V)\in\ell_{p,w}(\N_0;b^n)$ for some $p>1/2$. Then
\begin{equation}\label{p<1}
     N_-(\BH_{\a V})\le
C\a^p\sup_{t>0}\left(t^p\sum_{\y_n(V)>t}b^n\right)
\end{equation}
and
\begin{equation}\label{p<1bis}
   \limsup_{\a\to\infty}\a^{-p}  N_-(\BH_{\a V})\le
C\limsup_{t\to 0}\left(t^p\sum_{\y_n(V)>t}b^n\right).
\end{equation}
In particular, if $\bsymb\y(V)\in\ell_{p,w}^\circ(\N_0,b^n))$, then
\eqref{osmall} is satisfied.

 If $\bsymb\y(V)\in\ell_p(\N_0,b^n)$, then along with \eqref{osmall}
we have
\begin{equation}\label{strong<1}
    N_-(\BH_{\a V})\le
C\a^p\sum_{n=0}^\infty\y^p_n(V)b^n.
\end{equation}

\vs 2$^\circ$ Conversely, if $N_-(\BH_{\a V})=O(\a^p)$ with some
$p\ge 1/2$, then $\bsymb\y(V)\in \ell_{p,w}(\N_0;b^n)$. If
$N_-(\BH_{\a V})=o(\a^p)$ with some $p> 1/2$, then
$\bsymb\y(V)\in \ell_{p,w}^\circ(\N_0;b^n)$.
\end{thm}

\section{Proofs}\label{prfs}
\subsection{Using the Birman -- Schwinger principle}\label{prnc}
The proofs of all theorems \ref{genthm} -- \ref{th:pb} are based
upon the Birman -- Schwinger principle, see
\cite{BS,BS1} for its exposition most convenient for our
purposes. In order to formulate it for the particular case we
need, let us consider the quadratic form
\begin{equation}\label{bb}
\ba_V[u]=\int_\G V(x)|u(x)|^2dx
\end{equation}
where $V\ge 0$ is a Hardy weight on $\CH^\circ(\G)$. Then the quadratic
form $\ba_V$ is bounded in this space, and therefore, it generates
in $\CH^\circ(\G)$ a bounded, self-adjoint and non-negative operator which we denote
$\BA_V^\circ$. The notation $A^\circ_{\G,V}$ was used for this
operator in \cite{NS}, see Eq. (3.7) there.

The following statement expresses the Birman -- Schwinger
principle for the case of operators on trees.
\begin{prop}\label{l-BSCH}
Let a Hardy weight $V$ on $\CH^\circ(\G)$ be such that the
operator $\BA^\circ_V$ is compact. Then for any $\a>0$ the
quadratic form \eqref{shrqf} is bounded below and closed in
$L_2(\G)$, the negative spectrum of the corresponding operator
$\BH_{\a V}$ is finite, and moreover,
\begin{equation}\label{birsh}
N_-(\BH_{\a V})=n(\a^{-1};\BA^\circ_V),\qquad\forall\a>0.
\end{equation}
\end{prop}
This proposition reduces estimates of $N_-(0;\BH_{\a
V})$ to the spectral analysis of the compact operator
$\BA^\circ_V$. Such analysis was carried out in \cite{NS}.  \vs

The equality \eqref{birsh} implies that for any $p>0$
\begin{equation}\label{pr:1}
\sup_{\a>0}\,\a^{-p}N_-(\BH_{\a V})=\sup_{s>0}
s^p\,n(s;\BA^\circ_V) =\|\BA^\circ_V\|^p_{p,w}
\end{equation}
and \[\limsup_{\a\to\infty}\,\a^{-p}N_-(\BH_{\a
V})=\limsup_{s\to0} s^p\,n(s;\BA^\circ_V)
=\inf_{\BT\in\CC_{p,w}^\circ}\|\BA^\circ_V-\BT\|^p_{p,w}.\]

\subsection{Proof of Theorems \ref{genthm}--\ref{th:p}.}\label{pr1}
Due to \eqref{pr:1}, the inequality \eqref{bas} turns into the
estimate (3.14) in Theorem 3.4 of the paper \cite{NS}. By the
inequality (3.13) in the same theorem, finiteness of the integral
in \eqref{narr} implies $\BA_V^\circ\in \CC_p$. Since
$\CC_p\subset\CC_{p,w}$, this immediately yields \eqref{narr}.

The inequality \eqref{lsup} follows from \eqref{bas}
automatically, see the proof of a similar statement in \cite{BS},
Section 4.3. Note also that below we give the detailed proof of the inequality
\eqref{lsup3} in \thmref{th:p}. The proof of \eqref{lsup} can be reconstructed
by analogy.
\thmref{special} is a particular case of \thmref{genthm} for a
special choice of the weight function, and it does not need a
separate proof.

\thmref{th:p}, except for the estimate \eqref{lsup3}, is a direct
consequence (actually, just a reformulation in the equivalent
terms) of Theorem 4.1 in \cite{NS}, for the regular trees and the
symmetric weights. One only has to take into account that in
\cite{NS} the operators $\BA_V$ were considered, rather than
$\BA^\circ_V$. The operator $\BA_V$ corresponds to the  quadratic
form \eqref{bb} on the space $\CH^{1,0}(\G)$. If the tree is
transient, this space is wider than $\CH^\circ(\G)$. For this
reason,
\begin{equation*}
    \|\BA^\circ_V\|_\CC\le \|\BA_V\|_\CC,
\end{equation*}
where $\|\cdot\|_\CC$ stands for the (quasi-)norm in any class of
operators we are dealing with. It immediately follows that all the
estimates derived in \cite{NS} for $\BA_V$, hold also for the
operator $\BA_V^\circ$.

To justify \eqref{lsup3}, we note that for any $V\in L_1(\G)$ with
compact support the estimate \eqref{1/2} holds. For the
corresponding operator $\BA_V^\circ$ this implies
\[ n(s;\BA_V^\circ)=O(s^{-1/2})=o(s^{-p})\] for any $p>1/2$. Given an $\e>0$,
we can find a finite subset $E_\e\subset\N$ such that

\[ \sup_{t>0}\left(t^p\sum_{n\notin E_\e:\y_n(V)>t}g_0(t_{n+1})\right)
\le \limsup_{t>0}\left(t^p\sum_{n\in\N_0
:\y_n(V)>t}g_0(t_{n+1})\right)+\e.\] Let $\chi_\e$ stand for the
characteristic function of the set $\cup_{n\in
E_\e}(t_n,t_{n+1})$. Take $V_\e(x)=V(x)\chi_\e(|x|)$ and
$V^\e(x)=V(x)-V_\e(x)$, then $\BA_V^\circ=\BA_{V_\e}^\circ+
\BA_{V^\e}^\circ$ and $n(s;\BA_{V_\e}^\circ)=o(s^{-p})$. It
follows from \eqref{<1} that
\begin{gather*}
\limsup_{s\to 0} s^p n(s;\BA_V^\circ)=\limsup_{s\to 0} s^p n(s;\BA_{V^\e}^\circ)\\ \le
\sup_{t>0}\left(t^p \sum_{\y_n(V^\e)>t}g_0(t_{n+1})\right)\le
\limsup_{t>0}\left(t^p\sum_{n\in\N_0 :\y_n(V)>t}g_0(t_{n+1})\right)+\e.
\end{gather*}
Since $\e>0$ is arbitrary, \eqref{lsup3} is justified.

\subsection{Proof of \thmref{th:pb}.}
The estimate \eqref{p<1} for $1/2<p<1$ is a part of \thmref{th:p}
above. Now we prove that if $\bsymb\y(V)\in\ell_\infty$, then the
operator $\BA_V^\circ$ is bounded, with the estimate
\begin{equation*}
\|\BA_V^\circ\|\le C \sup_n\y_n(v).
\end{equation*}
To check this, it is sufficient to show that $V(x)$ is a Hardy
weight on $H^\circ(\G)$. Suppose first that $d<2$, so that the
tree is recurrent. The condition \eqref{hardy} is equivalent to
\[\int_t^\infty v(s)s^{d-1}ds\le Ct^{d-2},\qquad\forall t>0.\]
Since the numbers $t_n$ grow as a geometric progression, it is
sufficient to check this inequality for the values $t=t_n,\
n\in\N$. We have
\begin{gather*}
\int_{t_n}^\infty v(s)s^{d-1}ds=\sum_{j=n}^\infty
\int_{t_j}^{t_{j+1}}v(s)s^{d-1}ds\\
\le\sum_{j=n}^\infty t_{j+1}^{d-2}\y_j(V)\le
\sup_j\left(\y_j(V)\sum_{j=n}^\infty t_{j+1}^{d-2}\right).
\end{gather*}
By \eqref{reg1}, the series in the last expression converges and its sum is
controlled by $t_n^{d-2}$. This gives the desired result in the recurrent case.
For the transient case $d>2$ the argument is similar.

The estimates \eqref{p<1} and \eqref{strong<1} follow from the
estimates for $p=1/2$ and $p=\infty$ by the same interpolation
argument as the one used in \cite{NS} for the proof of Theorem
4.1. For interpolation, it is necessary to extend the definition
of the operator $\BA_V^\circ$ to the case where $V(x)$ is an arbitrary measurable
function on $\G$, such that $|V(x)|$ is a Hardy weight. This extension is evident
and we give no detail; cf. Section 3.1 in \cite{NS}.

We interpolate the mapping $\Pi: v\mapsto \BA_V^\circ,\
V(x)=v(|x|)$. In both cases, one of the basic results for
interpolation is the one for $p=1/2$, that is the the estimate
\[ \|\BA^\circ_V\|_{1/2,w}\le C\|\bsymb\y(V)\|_{\ell_{1/2}}.\]
Recall that this is equivalent to \eqref{1/2}. Under the
assumptions of \thmref{th:pb}, due to the special character of the
tree and the function $V$, another basic result concerns the case
$p=\infty$, while for the general trees and arbitrary potentials
we could use only Theorem 3.3 in \cite{NS}, which corresponds to
the case $p=1$. This argument justifies the statement 1$^\circ$ of
\thmref{th:pb}. \vs

For justifying the statement 2$^\circ$, we use Theorem 4.6 and
Remark to Corollary 4.7 in \cite{NS}. Note that in this theorem
the properties of the operators $\BA_V$ on $\CH^{1,0}(\G)$ were
discussed. However, the construction in the proof of Theorem 4.6,
suggested in \cite{NS}, uses only functions with bounded support,
and hence the result applies also to the operators $\BA_V^\circ$.
Note also that Corollary 4.7 in \cite{NS} summarizes the results
of Theorems 4.1 and 4.6 of this paper. For the $b$-regular trees
the first of these results extends to all $p>1/2$, and this
automatically leads to extension of the Corollary 4.7 to all such
$p$.

\section{Neumann boundary condition at the root}\label{neum}
Here we discuss the changes in the above results, appearing if in \eqref{sch} we replace
the Dirichlet Laplacian $\D_\CD$ by the Neumann Laplacian
$\D_\CN$. In order to indicate the difference between these two
cases, below we denote the corresponding operators by $\BH_{\a
V;\CD}$ and $\BH_{\a V;\CN}$. It immediately follows from the
variational principle that
\[ N_-(\BH_{\a V;\CD})\le N_-(\BH_{\a V;\CN})\le N_-(\BH_{\a V;\CD})+1,\]
so that both quantities behave in a similar way as $\a\to\infty$.
The behavior in the weak coupling limit (i.e., as $\a\to 0$) may
differ. Indeed, if the tree $\G$ is recurrent, then for any
non-trivial potential $V\ge0$ the operator $\BH_{\a V;\CN}$ has at
least one negative eigenvalue. Hence, no estimate homogeneous  in
$\a$ is possible for $N_-(\BH_{\a V;\CN})$ in the recurrent case.

\vs Such estimates of $N_-(\BH_{\a V;\CN})$ are
possible for the transient trees. They are based upon an
analogue of Proposition \ref{l-BSCH} for the $\CN$-case. In order
to define an appropriate substitute for the operator
$\BA_V^\circ$, we need one more version of the homogeneous Sobolev
space on $\G$.\vs

Let $\G$ be a regular transient tree, and let $u\in H^1(\G)$ be a
function with bounded support. Suppose also that $u$ is symmetric,
that is, $u(x)=\phi(|x|)$ where $\phi(t)$ is some function on
$[0,\infty)$. Then
\[ u(o)=-\int_0^\infty \phi'(t)dt=-\int_0^\infty \phi'(t)\sqrt{g_0(t)}
\frac{dt}{\sqrt{g_0(t)}},\] whence
\begin{equation}\label{u(o)}
    |u(o)|^2\le l(\G)\int_0^\infty|\phi'(t)|^2g_0(t)dt
    =l(\G)\int_\G|u'|^2dx.
\end{equation}
Here $l(\G)$ is the reduced height of the tree $\G$, see
\eqref{trans}. The resulting inequality extends to all $u\in
H^1(\G)$ with bounded support, since for the symmetric component $u_{sym}$ of $u$ we
always have $\|u'_{sym}\|\le \|u'\|$ by \cite{S}, Theorem 3.2
(applied to the function $u'$).

Define the space $\CH(\G)$ as the completion of $H^1(\G)$ in the norm
\[ \|u\|_\CH:=\|u'\|_{L_2(\G)}.\]
The inequality \eqref{u(o)} extends by continuity to all $u\in
\CH(\G)$, and it shows that $\CH(\G)$ can be considered as a
function space on $\G$. This is no more true if the tree is
recurrent: then the elements of such completion are the classes of
functions, differing by an arbitrary constant.\vs

Below we assume that $\G$ is a transient regular tree. We have
shown that then the linear functional $\g: u\mapsto u(o)$ on the
space $\CH(\G)$ is continuous, and
\begin{equation}\label{gam}
    \|\g\|\le\sqrt{l(\G)}.
\end{equation}
The space $\CH^\circ(\G)$, introduced in Section \ref{sob}, can be
realized as the subspace $\{u\in\CH(\G):\g u=0\}$. Clearly,
\begin{equation*}
    \dim \CH(\G)/\CH^\circ(\G)=1.
\end{equation*}

 By the Riesz theorem, there is a unique
function $h\in \CH(\G)$, such that
\[ u(o)=(u,h)_{\CH(\G)}=\int_\G u'\overline{h'}dx,\qquad\forall
u\in\CH(\G).\] It is easy to see that this function is symmetric, and is given by
\begin{equation*}
    h(x)=h_0(|x|)\ {\text{ where }}\ h_0(t)=\int_t^\infty\frac{ds}{g_0(s)}.
\end{equation*}
Actually, $h(t)$ is the harmonic function on $\G$, vanishing at
infinity and such that $h(o)=h_0(0)=l(\G)$. Its norm in $\CH(\G)$
is given by
\[ \|h\|^2_{\CH}=\int_\G|h'(x)|^2dx=\int_0^\infty{h_0'(t)}^2
g_0(t)dt=l(\G).\] This shows that in \eqref{gam} we actually have
the equality sign.\vs

Let a function $V(x)\ge 0$ on a transient regular tree $\G$ be
such that the quadratic form \eqref{bb} is bounded in the space
$\CH(\G)$. Then it defines a bounded operator in this space, say,
$\wt\BA_V$. It also defines the bounded operator $\BA^\circ_V$ in
its subspace $\CH^\circ(\GH)$, see Section \ref{prnc}. Our next
goal is to compare their (quasi-)norms in various spaces of
operators.

\begin{thm}\label{N}
Let $\G$ be a transient regular tree, and let $V\in L_1(\G),\ V\ge
0$. Suppose that the operator $\BA^\circ_V$ belongs to some space
$\CC$, where $\CC=\CC_p,\CC_{p,w}$, or $\CC_{p,w}^\circ$. Then the
operator $\wt\BA_V$ belongs to the same space, and, moreover,
\begin{equation}\label{n-tr}
    \|\wt\BA_V\|_\CC\le
    C\left(\|\BA_V^\circ\|_\CC+t_1\int_{e_0}Vdx\right).
\end{equation}
where $e_0$ stands for the edge of $\G$ emanating from the root,
and the constant $C$ depends only on the reduced height $l(\G)$
and on the length of $e_0$ (that is, on the number $t_1$).
\end{thm}
\begin{proof}
Define a function $\phi$ on $\G$ in the following way. Let
$\phi(x)=0$ everywhere outside the edge $e_0$. We identify $e_0$
 with the segment $[0,t_1]$ and take
 $\phi(x)=t_1^{-1}(t_1-x),\ x\in e_0$. Given a function $u\in \CH^1(\G)$,
 we decompose it as
 \[ u(x)=u_0(x)+u(o)\phi(x),\]
 then $u_0\in \CH^{\circ}(\G)$. We have
\begin{gather}
 \int_\G V(x)|u(x)|^2dx\le 2\left(\int_\G V(x)|u_0(x)|^2dx+
|u(o)|^2\int_{e_0}V(x)|\phi(x)|^2dx\right)\notag\\
\le 2\left(\int_\G V(x)|u_0(x)|^2dx+|u(o)|^2\int_{e_0}V(x)dx\right).\label{es}
\end{gather}
 Further, we have $\|\phi'\|^2=t_1^{-1}$, and hence,
\[\|u_0'\|\le \|u'\|+|u(o)|\|\phi'\|\le\|u'\|\left(1+\sqrt{l(G)t_1^{-1}}\right).\]
It follows that
\begin{equation}\label{es1}
\|u_0'\|^2+|u(o)|^2\|\phi'\|^2\le \|u'\|^2(2+3l(\G)t_1^{-1}).
\end{equation}

The inequalities \eqref{es} and \eqref{es1} show that the (quasi-)norm $\|\wt\BA_V\|_\CC$ is controlled by the sum of two terms. The first is $\|\BA_V^\circ\|_\CC$ and the second is the
  (quasi-)norm in $\CC$ of the operator of multiplication by the number
$M=\int_{e_0}Vdx\|\phi'\|^{-2}=t_1\int_{e_0}Vdx=\y_0(V)$, acting in the one-dimensional space,
generated by the function $\phi(x)$.
 This results in the estimate \eqref{n-tr}.
\end{proof}

An analogue of Proposition \ref{l-BSCH} for the $N$-case shows that
\[ N_-(\BH_{\a V;\CN})=n(\a^{-1};\wt\BA_V).\]
As a result of this equality and \thmref{N}, we come to the following general conclusions.

\begin{thm}\label{genthm-N}
Let $\G$ be a regular, transient metric tree and let $V\in L_1(\G),\ V\ge 0$.
Then

1$^\circ$ Each estimate \eqref{bas}, \eqref{narr}, \eqref{bas1}, \eqref{narr1}, \eqref{<1},
\eqref{str<1}, \eqref{p<1}, \eqref{strong<1}, with an additional term $C\a^p t_1\int_{e_0}Vdx$ in the right-hand side,
holds for the operator $\BH_{\a V;\CN}$. If $p=1/2$, the same is true for the estimate \eqref{1/2},
and the asymptotic formula \eqref{w} remains valid.

2$^\circ$ the inequalities \eqref{lsup}, \eqref{lsup1}, \eqref{lsup3}, and \eqref{p<1bis},
remain valid for the operator $\BH_{\a V;\CN}$. In particular, the finiteness of the right-hand sides in \eqref{narr}, \eqref{narr1}, \eqref{str<1}, and \eqref{strong<1} implies \eqref{osmall}.

3$^\circ$ The statement 2$^\circ$ of \thmref{th:pb} remains valid for the operator $\BH_{\a V;\CN}$.
\end{thm}

We also present a simple but useful result which follows from \thmref{N} and Theorem 3.3 in \cite{NS}.
\begin{thm}\label{tr}
Let $\G$ be a transient regular tree, and let $V\ge 0$ be a measurable function on $\G$. The operator $\wt\BA_V$
is trace class if and only if $\int_\G|x|V(x)dx<\infty$, and
\begin{equation}\label{trm}
\|\wt\BA_V\|_1=\tr\wt\BA_V\le C\left(\int_\G|x|V(x)dx+t_1\int_{e_0}Vdx\right)
\end{equation}
where $C$ is the constant from \eqref{n-tr}.
\end{thm}
It follows from this theorem that under its assumptions we have
\begin{equation}\label{tr1}
N_-(\BH_{\a V;\CN})\le C \a\left(\int_\G|x|V(x)dx+t_1\int_{e_0}Vdx\right),
\end{equation}
and also $N_-(\BH_{\a V;\CN})=o(\a)$.

\section{Concluding remarks}\label{conc}
\subsection{On the sharpness of \thmref{special}.}\label{sharp}
Below we present an example in which both Theorems \ref{special}
and \ref{th:pb} apply. Then, based upon the statement 2$^\circ$ of
the latter, we derive the sharpness of the first theorem.

\begin{example}\label{ex}
Let $\G$ be a $b$-regular tree with some $b>1$, and let $t_n=b^\frac{n}{d-1}$. Then $\G$ has global dimension $d$, and we assume $d>2$. On $\G$ we consider the symmetric potential
$V(x)=v(|x|)$ where
\[ v_(t)=t_{n+1}^{-2}b^{-\frac{n}{p}},\qquad t_n<t\le t_{n+1}.\]
Then
\[\y_n(V)=t_{n+1}^{-2}(t_{n+1}-t_n)b^{-\frac{n}{p}}=cb^{-\frac{n}{p}}, \qquad c=1-b^{-\frac1{d-1}}.\]

Let us check the assumptions of \thmref{th:pb}, 1$^\circ$. For small $t>0$ the sum
\[ \sum_{cb^{-\frac{n}{p}}>t}b^n=\sum_{b^n<(t/c)^{-p}}b^n\]
has the order $O(t^{-p})$ but not $o(t^{-p})$. It follows from \thmref{th:pb} that
$N_-(\BH_{\a V})=O(\a^p)$ where $O$ cannot be replaced by $o$.

Now, assuming $p>1$, we apply \thmref{special} to the same
potential. A simple calculation shows that the supremum in the
right-hand side of \eqref{bas1} is finite, and hence, $N_-(\BH_{\a
V})=O(\a^p)$. The reference to  \thmref{th:pb} shows that the
estimate is sharp. As a consequence, we see that the estimate
\eqref{lsup1} in \thmref{special} is order-sharp.
\end{example}
\subsection{Comparison with the results of \cite{EFK}.}\label{compar}
We do not discuss here the general Theorem 2.4 in \cite{EFK} and restrict
ourselves to the comparison between its Corollary 2.6 and our
\thmref{special} -- more exactly, its analogue for the operator
$\BH_{\a V;\CN}$. For the symmetric potentials and $p>1$ the
corollary (where one should take $a=2p-1$) is equivalent to our estimate
\eqref{narr}, within the value of the constant factor. The authors of
\cite{EFK} overlook the possibility to extend their result to the
general (non-symmetric) potentials. For $p=1$ the result is covered by
our \thmref{tr}.
They also do not discuss the important problem of the sharpness of the
estimate in the strong coupling limit, and they have nothing similar to
our inequalities \eqref{bas1} and \eqref{lsup1}, where the weak
$L_p$-spaces are involved.

 The results of \cite{EFK} concern only the case $p\ge1$, while our
Theorem \ref{th:p} covers the case $1/2\le p<1$, and for the
$b$-regular trees \thmref{th:pb} gives the necessary and
sufficient conditions for $N_-(\BH_{\a V})=O(\a^p)$ with an
arbitrary $p> 1/2$.\vs

We would like to emphasize that this critics does not concern other
parts of the
paper \cite{EFK}, including a material on the so-called
Lieb -- Thirring estimates on trees.

\end{document}